\documentclass{amsart}
\usepackage{graphicx} 
\usepackage{hyperref}
\usepackage[alphabetic]{amsrefs}
\newtheorem{theorem}{Theorem}

\title{The $\theta$ invariant recovers the Rozansky-Overbay invariant}
\author{Ramana Murugesan}
\date{March 2026}
\address{Department of Mathematics\\
  University of Toronto\\
  Toronto Ontario M5S 2E4\\
  Canada}
\email{ramana@math.toronto.edu}

\begin{document}

\maketitle

Let $K$ be a knot. \cite{pai} wrote formulas for the Rozansky-Overbay knot
invariant $\rho_1(K)\in\mathbf C[T^{\pm1}]$ defined in \cites{Ro1,Ro2,Ro3,Ov}.
Later, \cite{theta} define the invariant $\theta(K)\in\mathbf C[T_1^{\pm1},T_2^{\pm1}]$.
In this paper we show that the $\theta$ invariant is stronger than the Rozansky-Overbay invariant,
answering \cite[Conjecture 24]{theta}:
\begin{theorem}\label{thm:main}
    Let $K$ be a knot. Then $\rho_1(K)=-\theta(K)|_{T_1\leftarrow 1,T_2\leftarrow T}$
\end{theorem}

We begin by recalling one definition of the Alexander polynomial given in \cite{pai}.
Consider
an upright knot diagram $D$, which is an oriented long knot diagram so that at each crossing,
both overlapping strands point upwards. Every knot admits such a projection, since
crossings can be rotated to make them upright. Suppose we have such a diagram with
$n$ crossings. We start at the initial edge and label the edges $k=1,\ldots,2n+1$ along
the orientation of the knot.
    For each edge $k$, its \textbf{geometric rotation number} $\varphi_k$ is defined as the signed
    number of times that the tangent to $k$ is horizontal, where cups have positive sign and
    caps have negative sign.

Then, we denote each crossing $c$ by the tuple $(s,i,j)$ consisting of
the sign of $c$, the incoming 
overcrossing edge, and the incoming undercrossing edge respectively, and
define $X=\{c=(s,i,j):\text{ $c$ crossing}\}$. Letting $T$ be an indeterminate,
define the $(2n+1)\times(2n+1)$ matrix of Laurent polynomials
\begin{equation}
    A=I-\sum_{c\in X} (T^s E_{i,i+1}+(1-T^s)E_{i,j+1}+E_{j,j+1}) \label{A}
\end{equation}
where $E_{pq}$ is the elementary matrix with entry $1$ at $(p,q)$ and $0$ elsewhere.

Note that the normalized Alexander
polynomial is given by \[\Delta(K):=T^{(-\varphi(D)-w(D))/2}\det(A)\] where
$\varphi(D)=\sum_k \varphi_k$ and $w(D)=\sum_c s_c$ (where $s_c$ is the sign of $c$.)

Now we let $G=(g_{\alpha\beta}):=A^{-1}$. For a crossing $c$ define
\[R_1(c)=s(g_{ji}(g_{j+1,j}+g_{j,j+1}-g_{ij})-g_{ii}(g_{j,j+1}-1)-1/2)\]
Then $\rho_1(D)\in\mathbf C[T^{\pm1}]$ is defined by 
\[\rho_1(D):=\Delta^2\left(\sum_c R_1(c)-\sum_k\varphi_k(g_{kk}-1/2)\right)\]
It is shown in \cite{pai} that this is invariant under the Reidemeister-type moves of
an upright knot diagram and so this defines a knot invariant $\rho_1(K)$.

Now to define the $\theta$ invariant, take further
indeterminates $T_1,T_2$, and $T_3=T_1T_2$ and let $\Delta_i:=\Delta|_{T\leftarrow T_i}$ and
$G_i=(g_{i\alpha\beta}):=G|_{T\leftarrow T_i}$. For all crossings $c,c_0,c_1$, and edges $k$, define
\begin{align*}
    F_1(c) &=\  s
    \left[1/2 - g_{3ii} +  T_2^s g_{1ii} g_{2ji} - T_2^s g_{3jj} g_{2ji}
      - (T_2^s-1) g_{3ii} g_{2ji} \right. \\
  \nonumber &\quad
    \left. + (T_3^s-1) g_{2ji} g_{3ji} - g_{1ii} g_{2jj} + 2 g_{3ii} g_{2jj}
      + g_{1ii} g_{3jj} - g_{2ii} g_{3jj} \right] \\
  \nonumber &\quad
  + \frac{s}{T_2^s-1}
    [
      (T_1^s-1)T_2^s \left( g_{3jj} g_{1ji} - g_{2jj} g_{1ji} + T_2^s g_{1ji} g_{2ji} \right) \\
  \nonumber &\quad
      + (T_3^s-1) g_{3ji} ( 1 - T_2^s g_{1ii} + g_{2ij} + (T_2^s-2) g_{2jj} 
         (T_1^s-1) (T_2^s+1) g_{1ji} ]\\
    F_2(c_0,c_1)&=\frac{s_1 (T_1^{s_0}-1) (T_3^{s_1}-1) g_{1j_1i_0} g_{3j_0i_1}}{T_2^{s_1}-1}(T_2^{s_0} g_{2i_1i_0}+g_{2j_1j_0} - T_2^{s_0} g_{2j_1i_0}
    -g_{2i_1j_0} )\\
    F_3(k)&=\varphi_k(g_{3kk}-1/2)
\end{align*}
Then $\theta(D)\in \mathbf C[T_1^{\pm1},T_2^{\pm1}]$ is defined by
\[\theta(D):=
\Delta_1\Delta_2\Delta_3\left(\sum_c F_1(c)+\sum_{c_0,c_1}F_2(c_0,c_1)+\sum_k F_3(k)\right)\]
Although it is only obvious from the formula that $\theta(D)$ is a rational
function, \cite{theta} show that the denominators of the form
$(T_2^s-1)$ in $F_2$ vanish in the summation above, so $\theta(D)$ is a Laurent polynomial. They also show
that $\theta(D)$ is invariant under Reidemeister-type moves, and so defines a knot invariant
$\theta(K)$.

We now prove the theorem.
\begin{proof}[Proof of Theorem \ref{thm:main}]
    Let $D$ be a upright knot diagram representing $K$.
    Since $T_3=T_1T_2=T$ we have
\begin{align*}g_{ii}&=g_{2ii}=g_{3ii}\\
\Delta&=\Delta_2=\Delta_3\\
\Delta_1&=\Delta|_{T\leftarrow1}=1
\end{align*}
where the last equality comes from the normalization of $\Delta$. Since $F_2\equiv0$
and $F_3(k)=\varphi_k(g_{kk}-1/2)$ 
it suffices to show that $R_1(c)=-F_1(c)$ for each $c=(s,i,j)$.
With the above equations,
\begin{equation*}
    F_1(c)=\frac{s}{2}[1+2(-1+T^s)g^2_{ji}+2g_{ji}(1+g_{ij}-2g_{jj})+2g_{ii}(-1-(-1+T^s)g_{ji}+g_{jj})] \label{eq:f1}
\end{equation*}

To show that this equals $-R_1(c)$, we need to eliminate the
$g_{j,j+1}$ and $g_{j+1,j}$ terms in $R_1(c)$.
Recall that $G=A^{-1}$. Multiplying \eqref{A} by $G$,
\begin{align*}
1&=(AG)_{jj}=g_{jj}-g_{j+1,j}\\
0&=(GA)_{j,j+1}=g_{j,j+1}-(1-T^s)g_{ji}-g_{jj}\end{align*}
Thus:
\begin{align*}
    R_1(c)&=s\left[\vphantom{\frac12}g_{ji}((g_{jj}-1)+(g_{jj}+(1-T^s)g_{ji})-g_{ij})\right.\\
    &\quad-g_{ii}(g_{jj}
    \left.+(1-T^s)g_{ji}-1)-\frac{1}{2}\right]\\
        &= \frac{s}{2}[-1+2(1-T^s)g_{ji}^2+2g_{ii}(-g_{jj}-(1-T^s)g_{ji}+1)\\
        &\quad+2g_{ji}(2g_{jj}-1-g_{ij})]\\
        &= -F_1(c)
\end{align*}
\end{proof}
\section*{Acknowledgements}
The author thanks Dror Bar-Natan for suggesting this problem and helpful feedback.
\bibliography{refs}
\end{document}